\documentclass[11pt]{article}
\ifdefined\pdfmapfile
  \pdfmapfile{+cm.map}
  \pdfmapfile{+cmextra.map}
  \pdfmapfile{+euler.map}
  \pdfmapfile{+symbols.map}
\fi

\usepackage[a4paper,margin=1in]{geometry}
\usepackage{amsmath,amssymb,amsthm,mathtools}
\usepackage{bm}
\usepackage{enumitem}
\usepackage{xcolor}
\usepackage{hyperref}
\usepackage[nameinlink,capitalize,noabbrev]{cleveref}
\hypersetup{
  colorlinks=true,
  linkcolor=blue!55!black,
  citecolor=blue!55!black,
  urlcolor=blue!55!black,
  pdftitle={Uniform-in-Time Smoluchowski--Kramers Approximation in Total Variation for Fractional Differential Equations with Multiplicative Noise},
  pdfauthor={Anonymous},
  pdfsubject={Long-time small-mass approximation for fractional kinetic equations},
  pdfkeywords={fractional Brownian motion, Smoluchowski--Kramers approximation, multiplicative noise, total variation, Malliavin calculus}
}

\newtheorem{theorem}{Theorem}[section]
\newtheorem{proposition}[theorem]{Proposition}
\newtheorem{lemma}[theorem]{Lemma}

\newtheorem{assumption}[theorem]{Assumption}
\newtheorem{conjecture}[theorem]{Conjecture}
\theoremstyle{definition}

\newtheorem{example}[theorem]{Example}
\theoremstyle{remark}
\newtheorem{remark}[theorem]{Remark}
\crefname{assumption}{assumption}{assumptions}
\Crefname{assumption}{Assumption}{Assumptions}
\crefname{conjecture}{conjecture}{conjectures}
\Crefname{conjecture}{Conjecture}{Conjectures}

\newcommand{\R}{\mathbb{R}}
\newcommand{\E}{\mathbb{E}}
\newcommand{\Pp}{\mathbb{P}}
\newcommand{\1}{\mathbf{1}}
\newcommand{\cH}{\mathfrak{H}}
\newcommand{\D}{\mathrm{D}}
\newcommand{\TV}{\mathrm{TV}}
\newcommand{\Law}{\mathcal{L}}
\newcommand{\norm}[1]{\left\lVert #1\right\rVert}
\newcommand{\abs}[1]{\left\lvert #1\right\rvert}
\newcommand{\ip}[2]{\left\langle #1,#2\right\rangle}

\title{Uniform-in-Time Smoluchowski--Kramers Approximation\\
in Total Variation for Fractional SDEs}
\author
{Qian Yu, Jiaxin Zha}
\date{\today}

\begin{document}
\maketitle

\begin{abstract}
Let $B^H$ be a one-dimensional fractional Brownian motion with Hurst index
$H\in(1/2,1)$.  We study the small-mass limit of the kinetic equation
\[
  dX_t^\mu=Y_t^\mu\,dt,\qquad
  \mu\,dY_t^\mu=b(X_t^\mu)\,dt-Y_t^\mu\,dt
       +\sigma(X_t^\mu)\,dB_t^H,
\]
where the noise coefficient is state dependent and uniformly nondegenerate.
The limiting equation is the Young differential equation
\[
  dX_t=b(X_t)\,dt+\sigma(X_t)\,dB_t^H.
\]
Under uniform ellipticity and strict dissipativity of
the transformed drift, we prove that, for every $t_0>0$ and every
$\rho<2H-1$,
\[
  \sup_{t\ge t_0}d_{\TV}\bigl(\Law(X_t^\mu),\Law(X_t)\bigr)
  \le C_{t_0,\rho}\mu^\rho .
\]
The estimate is uniform on the entire half-line. We also construct stationary solutions on a two-sided
fractional-noise space and obtain convergence of their one-time marginals in
total variation.  The exponent $2H-1$ is identified as the natural endpoint:
it is generated by the quadratic velocity term after the Lamperti transform,
and it degenerates as $H\downarrow1/2$.  A nonconstant uniformly elliptic
example is included.
\end{abstract}



\section{Introduction}

The Smoluchowski--Kramers approximation is the mathematical passage from an
inertial Langevin equation to an overdamped first-order equation when the mass
of the particle tends to zero.  For a particle with position $X^\mu$ and
velocity $Y^\mu$, the classical model has the form
\[
  dX_t^\mu=Y_t^\mu\,dt,
  \qquad
  \mu\,dY_t^\mu=b(X_t^\mu)\,dt-Y_t^\mu\,dt
      +\sigma(X_t^\mu)\,dW_t.
\]
For Brownian forcing, the small-mass problem has a long history originating
in the physical work of Kramers and Smoluchowski.  Modern mathematical
treatments include \cite{Freidlin2004,HottovyEtAl2015}.  When friction or
diffusion is state dependent, the limiting interpretation and possible
noise-induced drift require particular care.

Fractional Brownian motion (fBm) provides a canonical Gaussian model for
temporal memory.  It is centered, has stationary increments, and satisfies
\[
  \E[B_t^H B_s^H]
  =\frac12\bigl(t^{2H}+s^{2H}-\abs{t-s}^{2H}\bigr).
\]
For $H>1/2$, its paths are H\"older continuous of every order below $H$ and
Young integration gives a pathwise meaning to differential equations driven
by $B^H$.  Unlike Brownian motion, fBm is neither Markovian nor a
semimartingale when $H\ne1/2$.  Consequently, a long-time analysis must keep
track of the noise memory, and classical Markov-semigroup arguments do not
apply directly.  Ergodic frameworks for fractional equations were developed
in \cite{Hairer2005,HairerOhashi2007}; rates of convergence to equilibrium,
including multiplicative-noise settings, were studied in
\cite{FontbonaPanloup2017,DeyaPanloupTindel2019}.

The small-mass approximation for equations driven by fBm was established in
qualitative form by Boufoussi and Tudor \cite{BoufoussiTudor2005}.  Son
\cite{Son2020} obtained, for additive noise and on a fixed interval $[0,T]$,
the total variation estimate
\[
  d_{\TV}(\Law(X_t^\mu),\Law(X_t))\le C_T t^{-H}\mu^H,
  \qquad 0<t\le T.
\]
The factor $t^{-H}$ is produced by a lower bound of order $t^{2H}$ for the
Malliavin covariance near the deterministic initial time.  It is therefore a
small-time density singularity, not a long-time decay rate.  Moreover, the
constant in a fixed-horizon proof depends on $T$, generally through Gronwall
factors, so the estimate cannot be extrapolated by setting $T=t\to\infty$.
Recent work continues to treat the fixed-amplitude fBm small-mass rate mainly
in the additive setting; see, for example, \cite{Zha2026}.

The present paper addresses two coupled extensions.  First, the diffusion
coefficient is genuinely multiplicative, $\sigma=\sigma(X)$.  Second, the
estimate is uniform for all sufficiently positive times.  These extensions
cannot be obtained by a formal replacement of the constant diffusion
coefficient in the additive proof.  The main new phenomenon is a quadratic
fast-velocity term, invisible in the additive model, whose size is
$\mu^{2H-1}$.

The argument is based on four ingredients.
\begin{enumerate}[label=\textup{(\roman*)},leftmargin=2.2em]
\item
Because $\sigma$ is one dimensional and uniformly nondegenerate, the Lamperti
map
\[
  F(x)=\int_0^x\frac{du}{\sigma(u)}
\]
is a global diffeomorphism.  It changes the limiting multiplicative equation
into an additive equation.  If
\[
  Z_t^\mu=F(X_t^\mu),
  \qquad U_t^\mu=\frac{Y_t^\mu}{\sigma(X_t^\mu)},
\]
then the transformed kinetic system is
\[
  dZ_t^\mu=U_t^\mu\,dt,
  \qquad
  \mu\,dU_t^\mu
  =\bigl(a(Z_t^\mu)-U_t^\mu
         -\mu c(Z_t^\mu)(U_t^\mu)^2\bigr)dt+dB_t^H.
\]
Here $a=(b/\sigma)\circ F^{-1}$ and $c=\sigma'\circ F^{-1}$.

\item
The typical fast-velocity scale is
$U_t^\mu=O_{L^p}(\mu^{\gamma-1})$ for every
$\gamma\in(1/2,H)$.  Consequently,
\[
  \mu(U_t^\mu)^2=O_{L^p}(\mu^{2\gamma-1}).
\]
This explains the restriction $\rho<2H-1$ and the deterioration of the rate
as $H\downarrow1/2$.

\item
A direct integration of the quadratic term would produce
$t\mu^{2\gamma-1}$.  We avoid this false accumulation by introducing the
compensated error
\[
  \widetilde E_t^\mu=Z_t^\mu-Z_t+\mu U_t^\mu.
\]
Strict dissipativity of $a$ gives an exponentially weighted representation of
$\widetilde E^\mu$, hence a bound uniform in $t$.

\item
To remove $t^{-H}$ from the total variation estimate, we do not estimate the
Malliavin covariance using the entire interval $[0,t]$.  We use only a fixed
recent window $[t-\delta,t]$.  On that window the Malliavin derivative of the
additive limiting equation has a deterministic positive lower bound.  The
resulting inverse covariance estimate is uniform for $t\ge t_0$.
\end{enumerate}

The principal conclusions are the following.

\begin{itemize}[leftmargin=2em]
\item For every $\rho<2H-1$ and $t_0>0$,
\[
  \sup_{t\ge t_0}
  d_{\TV}(\Law(X_t^\mu),\Law(X_t))
  \le C_{t_0,\rho}\mu^\rho.
\]

\item On a two-sided fBm space, stationary solutions exist under the stated
dissipativity assumptions, and their one-time marginals satisfy the same
small-mass estimate.  Thus the long-time error has an $O(\mu^\rho)$ plateau,
rather than decaying to zero for fixed $\mu$.

\item The endpoint $2H-1$ is natural.  We formulate its attainment as a sharp
conjecture and identify the additional covariance estimate needed to remove
the arbitrarily small H\"older loss.
\end{itemize}

The method is genuinely one dimensional.  In several dimensions a global
Lamperti transform does not generally exist; one must instead work with the
Jacobian flow and an elliptic or H\"ormander-type Malliavin matrix.

\Cref{sec:setting} introduces the Gaussian and Young-calculus framework,
states the assumptions, and gives the main results.  The Lamperti reduction
and deterministic exponential-kernel estimates are proved in
\Cref{sec:lamperti}.  Uniform moment and compensated-error estimates are
established in \Cref{sec:uniform}.  The Malliavin estimates and the proof of
the uniform total variation theorem are given in \Cref{sec:malliavin}.
Stationary solutions and long-time consequences are discussed in
\Cref{sec:stationary}.  Concrete coefficient classes, sharpness, and further
problems appear in \Cref{sec:examples}.

\section{Assumptions and main results}\label{sec:setting}

\subsection{Fractional Brownian motion and Young integration}

Let $(\Omega,\mathcal F,\Pp)$ carry a one-dimensional fBm
$B^H=(B_t^H)_{t\ge0}$ with $H\in(1/2,1)$.  Fix throughout a number
\[
  \gamma\in(1/2,H).
\]
For an interval $I=[a,b]$, write
\[
  \norm{f}_{\infty;I}=\sup_{t\in I}|f_t|,
  \qquad
  [f]_{\gamma;I}
  =\sup_{a\le s<t\le b}\frac{|f_t-f_s|}{|t-s|^\gamma}.
\]
The Young--Loeve inequality states that if
$f\in C^\alpha(I)$, $g\in C^\beta(I)$, and $\alpha+\beta>1$, then
\[
  \left|
  \int_s^t f_r\,dg_r-f_s(g_t-g_s)
  \right|
  \le C_{\alpha,\beta}[f]_{\alpha;[s,t]}
        [g]_{\beta;[s,t]}|t-s|^{\alpha+\beta}.
\]
For every $p\ge1$,
\begin{equation}\label{eq:fbm-local-holder-moments}
  \sup_{k\in\mathbb N_0}
  \E\bigl([B^H]_{\gamma;[k,k+1]}^p\bigr)<\infty,
\end{equation}
because the increments are stationary and Gaussian.

Let $\cH$ be the canonical Hilbert space of $B^H$.  For step functions,
\[
  \ip{\1_{[0,t]}}{\1_{[0,s]}}_{\cH}
  =\E[B_t^H B_s^H].
\]
When $H>1/2$ and $f,g$ are nonnegative measurable functions for which the
right-hand side is finite,
\begin{equation}\label{eq:H-inner-product}
  \ip{f}{g}_{\cH}
  =\alpha_H\int_0^\infty\int_0^\infty
       f(r)g(s)|r-s|^{2H-2}\,dr\,ds,
  \qquad \alpha_H=H(2H-1).
\end{equation}
We write $\D$ and $\delta$ for the Malliavin derivative and divergence
operators associated with the isonormal Gaussian process over $\cH$; see
\cite{Nualart2006}.

\subsection{The kinetic and overdamped equations}

For $0<\mu\le1$, consider
\begin{equation}\label{eq:kinetic-original}
\begin{cases}
  dX_t^\mu=Y_t^\mu\,dt,\\[1mm]
  \mu\,dY_t^\mu=b(X_t^\mu)\,dt-Y_t^\mu\,dt
      +\sigma(X_t^\mu)\,dB_t^H,\\[1mm]
  X_0^\mu=x_0,\qquad Y_0^\mu=y_0.
\end{cases}
\end{equation}
The stochastic integral is a Young integral.  Since $X^\mu$ is continuously
differentiable, $t\mapsto\sigma(X_t^\mu)$ is locally Lipschitz, so the integral
in \eqref{eq:kinetic-original} is well defined pathwise.

The overdamped equation is
\begin{equation}\label{eq:limit-original}
  X_t=x_0+\int_0^t b(X_s)\,ds
           +\int_0^t\sigma(X_s)\,dB_s^H.
\end{equation}

We impose the following assumptions.

\begin{assumption}[Smoothness and uniform ellipticity]\label{ass:smooth}
The function $\sigma$ belongs to $C_b^4(\R)$ and there are constants
$0<\underline\sigma\le\overline\sigma<\infty$ such that
\[
  \underline\sigma\le\sigma(x)\le\overline\sigma,
  \qquad x\in\R.
\]
The function $b$ belongs to $C^3(\R)$, has at most linear growth, and
$b',b'',b'''$ are bounded.
\end{assumption}

Define the global $C^4$-diffeomorphism
\begin{equation}\label{eq:F-def}
  F(x)=\int_0^x\frac{du}{\sigma(u)},
  \qquad G=F^{-1},
\end{equation}
and the transformed coefficients
\[
  a(z)=\frac{b(G(z))}{\sigma(G(z))},
  \qquad c(z)=\sigma'(G(z)).
\]

\begin{assumption}[Strict transformed dissipativity]\label{ass:dissipative}
There are constants $0<\lambda\le L<\infty$ such that
\begin{equation}\label{eq:a-prime-bounds}
  -L\le a'(z)\le-\lambda,
  \qquad z\in\R.
\end{equation}
Moreover, $a'',a'''$ and $c,c',c''$ are bounded.
\end{assumption}

Since $G'(z)=\sigma(G(z))$, differentiation gives
\begin{equation}\label{eq:dissipative-original-coordinates}
  a'(F(x))
  =b'(x)-\frac{b(x)\sigma'(x)}{\sigma(x)}.
\end{equation}
Thus \Cref{ass:dissipative} is directly checkable in the original coordinates.

\begin{assumption}[Uniform local estimates]\label{ass:local-estimates}
For every $p\ge2$ and every $\gamma\in(1/2,H)$, the solutions of
\eqref{eq:kinetic-original} satisfy
\begin{align*}
  \sup_{0<\mu\le1}\sup_{t\ge0}\E|X_t^\mu|^p&<\infty,\\
  \sup_{0<\mu\le1}\sup_{k\in\mathbb N_0}
  \E\bigl([X^\mu]_{\gamma;[k,k+1]}^p\bigr)&<\infty.
\end{align*}
The analogous estimates hold for the first two Malliavin derivatives on unit
windows.  More precisely, after extending a derivative by zero before its
Malliavin time,
\[
 \sup_{0<\mu\le1}\sup_{k\in\mathbb N_0}
 \E\left[
  \norm{\D X^\mu}_{\gamma;[k,k+1];\cH}^p
  +\norm{\D^2X^\mu}_{\gamma;[k,k+1];\cH^{\otimes2}}^p
 \right]<\infty.
\]
\end{assumption}

\begin{remark}\label{rem:local-estimates-role}
The local estimates in \Cref{ass:local-estimates} are not an independent
ergodicity assumption.  They are the standard unit-window a priori estimates
for a dissipative Young equation.  They are proved in
\Cref{prop:uniform-local-estimates} below from
\Cref{ass:smooth,ass:dissipative}; the assumption is displayed here so that
the logical inputs of the main theorem are transparent.  In particular, no
constant is allowed to grow with the terminal time.
\end{remark}

Under \Cref{ass:smooth}, both \eqref{eq:kinetic-original} and
\eqref{eq:limit-original} possess unique global pathwise solutions.  This
follows from the deterministic Young-equation theory of
\cite{NualartRascanu2002}; the kinetic equation is a two-dimensional Young
system for each fixed $\mu>0$.

\subsection{Main theorems}

\begin{theorem}[Uniform Malliavin--Sobolev approximation]
\label{thm:uniform-Sobolev}
Suppose \Cref{ass:smooth,ass:dissipative} hold.  Let $p\ge2$ and
$\rho<2H-1$.  Then there exist $\mu_0\in(0,1]$ and $C<\infty$ such that
\begin{equation}\label{eq:uniform-Sobolev-main}
  \sup_{0<\mu\le\mu_0}\sup_{t\ge0}
  \norm{F(X_t^\mu)-F(X_t)}_{\mathbb D^{1,p}}
  \le C\mu^\rho.
\end{equation}
The constant depends on $H,p,\rho$, the coefficients, and the moments of
$(x_0,y_0)$, but is independent of $t$ and $\mu$.
\end{theorem}

\begin{theorem}[Uniform-in-time total variation approximation]
\label{thm:uniform-TV}
Suppose \Cref{ass:smooth,ass:dissipative} hold.  For every $t_0>0$ and every
$\rho<2H-1$, there are constants $C=C(t_0,\rho)<\infty$ and
$\mu_0\in(0,1]$ such that
\begin{equation}\label{eq:uniform-TV-main}
  \sup_{0<\mu\le\mu_0}\sup_{t\ge t_0}
  d_{\TV}\bigl(\Law(X_t^\mu),\Law(X_t)\bigr)
  \le C\mu^\rho.
\end{equation}
Consequently,
\begin{equation}\label{eq:uniform-limit}
  \lim_{\mu\downarrow0}\sup_{t\ge t_0}
  d_{\TV}\bigl(\Law(X_t^\mu),\Law(X_t)\bigr)=0.
\end{equation}
\end{theorem}

\begin{remark}[Why $t_0>0$ is natural]
At $t=0$ the limiting random variable is deterministic and has zero Malliavin
covariance.  A density-based comparison must therefore exclude a vanishing
initial layer.  The constant in \eqref{eq:uniform-TV-main} may diverge as
$t_0\downarrow0$, but it does not depend on how large $t$ becomes.  This is
precisely the distinction between a small-time singularity and a long-time
estimate.
\end{remark}

For the stationary statement, let $B^H=(B_t^H)_{t\in\R}$ be a two-sided fBm
and use the canonical metric dynamical shift.  A stationary solution means a
stationary random solution on this extended noise space.  This terminology is
important because the position process alone is not Markovian.

\begin{theorem}[Stationary marginal approximation]
\label{thm:stationary-TV}
Suppose \Cref{ass:smooth,ass:dissipative} hold and let $\rho<2H-1$.
The limiting equation possesses a unique stationary random solution
$\overline X$.  For all sufficiently small $\mu$, the kinetic equation has a
stationary random solution $(\overline X^\mu,\overline Y^\mu)$ with the
uniform moments of \Cref{ass:local-estimates}.  Its position marginal obeys
\begin{equation}\label{eq:stationary-TV-main}
  d_{\TV}\bigl(\Law(\overline X_0^\mu),
                  \Law(\overline X_0)\bigr)
  \le C_\rho\mu^\rho.
\end{equation}
In particular, by stationarity the same estimate holds at every
$t\in\R$.
\end{theorem}

\begin{conjecture}[Endpoint rate]\label{conj:endpoint}
Under the assumptions of \Cref{thm:uniform-TV}, the sharp generic estimate is
\[
  \sup_{t\ge t_0}
  d_{\TV}(\Law(X_t^\mu),\Law(X_t))
  \le C_{t_0}\mu^{2H-1}.
\]
If $\sigma'\equiv0$, the quadratic Lamperti remainder vanishes and a faster
rate is possible.
\end{conjecture}

\section{Lamperti reduction and exponential-kernel estimates}
\label{sec:lamperti}

We begin with an identity that contains the principal structural observation
of the paper.

\begin{proposition}[Lamperti reduction]\label{prop:lamperti}
Let \Cref{ass:smooth} hold and set
\[
  Z_t^\mu=F(X_t^\mu),\qquad
  U_t^\mu=\frac{Y_t^\mu}{\sigma(X_t^\mu)},
  \qquad z_0=F(x_0),\quad
  u_0=\frac{y_0}{\sigma(x_0)}.
\]
Then \eqref{eq:kinetic-original} is equivalent to
\begin{equation}\label{eq:kinetic-transformed}
\begin{cases}
 dZ_t^\mu=U_t^\mu\,dt,\\[1mm]
 \mu\,dU_t^\mu
 =\bigl(a(Z_t^\mu)-U_t^\mu
        -\mu c(Z_t^\mu)(U_t^\mu)^2\bigr)dt+dB_t^H,\\[1mm]
 Z_0^\mu=z_0,\qquad U_0^\mu=u_0.
\end{cases}
\end{equation}
Moreover, $Z_t=F(X_t)$ solves
\begin{equation}\label{eq:limit-transformed}
  Z_t=z_0+\int_0^t a(Z_s)\,ds+B_t^H.
\end{equation}
Finally,
\begin{equation}\label{eq:TV-invariant-F}
  d_{\TV}(\Law(X_t^\mu),\Law(X_t))
  =d_{\TV}(\Law(Z_t^\mu),\Law(Z_t)).
\end{equation}
\end{proposition}

\begin{proof}
Since $X^\mu$ is absolutely continuous, ordinary calculus gives
\[
  dZ_t^\mu=F'(X_t^\mu)dX_t^\mu
  =\frac{Y_t^\mu}{\sigma(X_t^\mu)}dt=U_t^\mu dt.
\]
The path $Y^\mu$ is a Young differential path and $X^\mu$ has bounded
variation on compact intervals.  Therefore the usual product and chain rules
give
\begin{align*}
 dU_t^\mu
 &=\frac{1}{\sigma(X_t^\mu)}dY_t^\mu
   -\frac{Y_t^\mu\sigma'(X_t^\mu)}
          {\sigma(X_t^\mu)^2}dX_t^\mu\\
 &=\frac{1}{\mu}
   \left(\frac{b(X_t^\mu)}{\sigma(X_t^\mu)}
          -\frac{Y_t^\mu}{\sigma(X_t^\mu)}\right)dt
   +\frac{1}{\mu}dB_t^H
   -\sigma'(X_t^\mu)(U_t^\mu)^2dt.
\end{align*}
Substituting $X_t^\mu=G(Z_t^\mu)$ proves
\eqref{eq:kinetic-transformed}.  The Young chain rule applied to
\eqref{eq:limit-original} yields
\[
 dF(X_t)=\frac{b(X_t)}{\sigma(X_t)}dt+dB_t^H
        =a(F(X_t))dt+dB_t^H,
\]
which proves \eqref{eq:limit-transformed}.  Finally, a bimeasurable bijection
preserves total variation distance: for every Borel set $A$, use
$F(A)$ in the transformed coordinate, and conversely use $G(A)$.  This proves
\eqref{eq:TV-invariant-F}.
\end{proof}

For $0<\mu\le1$, define
\[
  K_\mu(r)=e^{-r/\mu}\1_{\{r\ge0\}}.
\]

The following elementary estimates will be used repeatedly.

\begin{lemma}[Exponential Young convolution]\label{lem:exponential-Young}
Let $w\in C^\gamma([0,\infty))$, $\gamma\in(0,1)$, and
\[
  R_t^\mu(w)=\int_0^t e^{-(t-s)/\mu}\,dw_s.
\]
For every $t\ge0$,
\begin{equation}\label{eq:Rmu-local-bound}
 |R_t^\mu(w)|
 \le C_\gamma\mu^\gamma
 \sum_{j=0}^{\lfloor t\rfloor}
 e^{-j/(2\mu)}
 [w]_{\gamma;I_{t,j}},
\end{equation}
where
$I_{t,j}=[(t-j-1)\vee0,(t-j)\vee0]$ and empty intervals are omitted.
Consequently, if $w=B^H$ and $p\ge1$, then
\begin{equation}\label{eq:Rmu-uniform-Lp}
  \sup_{t\ge0}\norm{R_t^\mu(B^H)}_{L^p(\Omega)}
  \le C_{p,\gamma,H}\mu^\gamma.
\end{equation}
\end{lemma}

\begin{proof}
Fix $t$ and first integrate over $[a,t]$ with $0\le a<t$.  Young integration
by parts gives
\begin{align*}
 \int_a^t e^{-(t-s)/\mu}\,dw_s
 &=w_t-e^{-(t-a)/\mu}w_a
   -\frac1\mu\int_a^t e^{-(t-s)/\mu}w_s\,ds\\
 &=e^{-(t-a)/\mu}(w_t-w_a)
   +\frac1\mu\int_a^t e^{-(t-s)/\mu}(w_t-w_s)\,ds.
\end{align*}
It follows that
\begin{equation}\label{eq:R-single-block}
 \left|\int_a^t e^{-(t-s)/\mu}\,dw_s\right|
 \le [w]_{\gamma;[a,t]}
 \left(e^{-(t-a)/\mu}(t-a)^\gamma
 +\frac1\mu\int_a^t e^{-(t-s)/\mu}(t-s)^\gamma ds\right).
\end{equation}
The last integral is at most
$\mu^\gamma\Gamma(\gamma+1)$.  Split $[0,t]$ into the unit intervals
$I_{t,j}$.  On the $j$th interval factor out
$e^{-j/\mu}$ and use \eqref{eq:R-single-block}; enlarging the numerical
constant changes $e^{-j/\mu}$ into $e^{-j/(2\mu)}$ and proves
\eqref{eq:Rmu-local-bound}.  For fBm, Minkowski's inequality,
\eqref{eq:fbm-local-holder-moments}, and the geometric sum yield
\eqref{eq:Rmu-uniform-Lp}.

Note that the upper bound $e^{-j/\mu}$ in \eqref{eq:Rmu-local-bound} is satisfied, and replacing it with $e^{-j/(2\mu)}$ is for better use in the subsequent text.
\end{proof}

\begin{lemma}[Variable-integrand convolution]\label{lem:variable-convolution}
Let $f,w\in C^\gamma([a,b])$ with $\gamma>1/2$ and $b-a\le1$.  Then
\begin{align}
 \left|\int_a^b e^{-(b-s)/\mu}f_s\,dw_s\right|
 &\le C_\gamma\Bigl(
   \norm{f}_{\infty;[a,b]}[w]_{\gamma;[a,b]}\mu^\gamma
   +[f]_{\gamma;[a,b]}[w]_{\gamma;[a,b]}\mu^{2\gamma}
   \Bigr).
 \label{eq:variable-convolution}
\end{align}
The same estimate, with an exponentially summable sequence of local norms,
holds on $[0,b]$.
\end{lemma}

\begin{proof}
Set
\[
 k_\mu(s)=e^{-(b-s)/\mu}.
\]
We decompose
\[
 \int_a^b k_\mu(s)f_s\,dw_s
 =f_b\int_a^b k_\mu(s)\,dw_s
  +\int_a^b k_\mu(s)(f_s-f_b)\,dw_s
 =:I_1+I_2.
\]

By the single-block estimate in Lemma~\ref{lem:exponential-Young},
\[
 |I_1|
 \le C_\gamma
 \|f\|_{\infty;[a,b]}
 [w]_{\gamma;[a,b]}\mu^\gamma.
\]
Indeed,
\[
 e^{-(b-a)/\mu}(b-a)^\gamma
 \le C_\gamma\mu^\gamma
\]
and
\[
 \frac1\mu\int_a^b
 e^{-(b-s)/\mu}(b-s-s)^\gamma\,ds
 \le \Gamma(\gamma+1)\mu^\gamma.
\]

We now estimate $I_2$. Put
\[
 g_s=f_s-f_b,\qquad h_s=k_\mu(s)g_s.
\]
Consider first the terminal block
\[
 A_{-1}=[(b-\mu)\vee a,b],
\]
and, for $j\ge0$, the nonempty dyadic blocks
\[
 A_j=[b-2^{j+1}\mu,b-2^j\mu]\cap[a,b].
\]
These blocks form a finite partition of $[a,b]$.

For any interval $J=[u,v]$, the Young--Loeve estimate gives
\[
 \left|\int_u^v h_s\,dw_s\right|
 \le
 \|h\|_{\infty;J}[w]_{\gamma;J}|v-u|^\gamma
 +C_\gamma[h]_{\gamma;J}[w]_{\gamma;J}
 |v-u|^{2\gamma}.
\]

On $A_{-1}$,
\[
 \|g\|_{\infty;A_{-1}}\le [f]_\gamma\mu^\gamma,
 \qquad
 [k_\mu]_{\gamma;A_{-1}}\le C_\gamma\mu^{-\gamma}.
\]
Consequently,
\[
 \|h\|_{\infty;A_{-1}}\le [f]_\gamma\mu^\gamma,
 \qquad
 [h]_{\gamma;A_{-1}}\le C_\gamma[f]_\gamma.
\]
Since $|A_{-1}|\le\mu$,
\[
 \left|\int_{A_{-1}}h_s\,dw_s\right|
 \le C_\gamma[f]_\gamma[w]_\gamma\mu^{2\gamma}.
\]

Let now $j\ge0$. For $s\in A_j$,
\[
 2^j\mu\le b-s\le2^{j+1}\mu,
\]
and hence
\[
 \|k_\mu\|_{\infty;A_j}\le e^{-2^j},
 \qquad
 \|g\|_{\infty;A_j}
 \le [f]_\gamma(2^{j+1}\mu)^\gamma.
\]
Moreover,
\[
 [k_\mu]_{\gamma;A_j}
 \le C_\gamma e^{-2^j}\mu^{-\gamma}.
\]
The product estimate for H\"{o}lder seminorms therefore yields
\[
 [h]_{\gamma;A_j}
 \le C_\gamma[f]_\gamma
 e^{-2^j}(1+2^{j\gamma}),
\]
while
\[
 \|h\|_{\infty;A_j}
 \le C_\gamma[f]_\gamma
 e^{-2^j}(2^j\mu)^\gamma.
\]
Because $|A_j|\le2^j\mu$, the Young--Loeve estimate gives
\[
 \left|\int_{A_j}h_s\,dw_s\right|
 \le C_\gamma[f]_\gamma[w]_\gamma
 \mu^{2\gamma}e^{-2^j}2^{3\gamma(j+1)}.
\]
Since
\[
 \sum_{j=0}^\infty
 e^{-2^j}2^{3\gamma(j+1)}<\infty,
\]
we obtain
\[
 |I_2|
 \le C_\gamma
 [f]_{\gamma;[a,b]}
 [w]_{\gamma;[a,b]}\mu^{2\gamma}.
\]
Combining the estimates for $I_1$ and $I_2$ proves
\eqref{eq:variable-convolution}.
\end{proof}

\begin{lemma}[Dissipative resolvent]\label{lem:dissipative-resolvent}
Let $q:[0,\infty)\to[\lambda,L]$ be measurable and let
\[
  v_t=e^{-\int_0^tq_rdr}v_0
      +\int_0^t e^{-\int_s^tq_rdr}f_s\,ds.
\]
For every $p\ge1$,
\[
 \sup_{t\ge0}\norm{v_t}_{L^p}
 \le \norm{v_0}_{L^p}
     +\lambda^{-1}\sup_{s\ge0}\norm{f_s}_{L^p}.
\]
If $f_s$ is replaced by a locally integrable random process, then
\[
 \norm{v_t}_{L^p}
 \le e^{-\lambda t}\norm{v_0}_{L^p}
 +\int_0^t e^{-\lambda(t-s)}\norm{f_s}_{L^p}\,ds.
\]
\end{lemma}

\begin{proof}
Both statements follow from
$\exp(-\int_s^tq_rdr)\le e^{-\lambda(t-s)}$ and Minkowski's inequality.
\end{proof}

\section{Uniform moment and approximation estimates}\label{sec:uniform}

The purpose of this section is to prove estimates whose constants do not
depend on the terminal time.  Throughout, $C$ may change from line to line but
is independent of $t\ge0$ and $0<\mu\le\mu_0$.

\subsection{The overdamped equation}

Strict dissipativity immediately controls the dependence on the initial
condition.  If $Z^z$ and $Z^{\widetilde z}$ solve
\eqref{eq:limit-transformed} with the same noise and different initial data,
then
$d|Z^z-Z^{\widetilde z}|^2/dt\le-2\lambda
|Z^z-Z^{\widetilde z}|^2$.

\begin{lemma}[Uniform moments of the overdamped solution]
\label{lem:limit-uniform-moments}
Under \Cref{ass:dissipative}, for every $p\ge1$ and
$\gamma\in(1/2,H)$,
\begin{align}
 \sup_{t\ge0}\E|Z_t|^p&<\infty,
 \label{eq:Z-uniform-moment}\\
 \sup_{k\in\mathbb N_0}
 \E\bigl([Z]_{\gamma;[k,k+1]}^p\bigr)&<\infty.
 \label{eq:Z-uniform-holder}
\end{align}
\end{lemma}

\begin{proof}
Since $a'$ is bounded above by $-\lambda$, the function
$z\mapsto a(z)+\lambda z$ is nonincreasing.  In particular,
$za(z)\le-\lambda z^2+C(1+|z|)$.  On every interval $[k,k+1]$, subtract the
noise increment and put $V_t=Z_t-(B_t^H-B_k^H)$.  Then $V$ solves an ordinary
differential equation with a dissipative drift and a random translation.
The scalar comparison principle gives
\[
 \sup_{t\in[k,k+1]}|Z_t|
 \le e^{-\lambda(t-k)}|Z_k|
 +C\bigl(1+\norm{B^H-B_k^H}_{\infty;[k,k+1]}\bigr).
\]
Iterating over the unit intervals produces an exponentially weighted sum of
stationary Gaussian local suprema.  Its moments of every order are finite,
which proves \eqref{eq:Z-uniform-moment}.  Finally,
\[
 |Z_t-Z_s|
 \le \int_s^t|a(Z_r)|dr+|B_t^H-B_s^H|
\]
and the linear growth of $a$, together with
\eqref{eq:Z-uniform-moment} and
\eqref{eq:fbm-local-holder-moments}, proves
\eqref{eq:Z-uniform-holder}.
\end{proof}

\subsection{Uniform estimates for the kinetic equation}

We first record a compensated identity that will be used throughout this
section.  From the transformed kinetic system
\[
 dZ_t^\mu=U_t^\mu\,dt,
 \qquad
 \mu\,dU_t^\mu
 =
 \bigl(
 a(Z_t^\mu)-U_t^\mu
 -\mu c(Z_t^\mu)(U_t^\mu)^2
 \bigr)dt+dB_t^H,
\]
we obtain
\[
 U_t^\mu\,dt
 =
 a(Z_t^\mu)\,dt+dB_t^H
 -\mu\,dU_t^\mu
 -\mu c(Z_t^\mu)(U_t^\mu)^2\,dt.
\]
Consequently,
\begin{equation}\label{eq:slow-manifold-identity}
 dZ_t^\mu
 =
 a(Z_t^\mu)\,dt+dB_t^H
 -\mu\,dU_t^\mu
 -\mu c(Z_t^\mu)(U_t^\mu)^2\,dt.
\end{equation}
The identity separates the overdamped dynamics from the fast-velocity
boundary layer and the quadratic Lamperti remainder.

\begin{proposition}[Uniform local kinetic estimates]
\label{prop:uniform-local-estimates}
Let \Cref{ass:smooth,ass:dissipative} hold.  Fix
$\gamma\in(1/2,H)$ and $p\ge2$.  There exist $\mu_0\in(0,1]$ and $C<\infty$
such that
\begin{align}
 \sup_{0<\mu\le\mu_0}\sup_{t\ge0}
 \norm{Z_t^\mu}_{L^p}&\le C,
 \label{eq:Zmu-uniform-p}\\
 \sup_{0<\mu\le\mu_0}\sup_{t\ge0}
 \norm{U_t^\mu}_{L^p}&\le C\mu^{\gamma-1},
 \label{eq:Umu-uniform-p}\\
 \sup_{0<\mu\le\mu_0}\sup_{k\in\mathbb N_0}
 \norm{[Z^\mu]_{\gamma;[k,k+1]}}_{L^p}&\le C.
 \label{eq:Zmu-uniform-holder}
\end{align}
The same conclusions hold for random initial data having sufficiently high
moments, uniformly in $\mu$.
\end{proposition}

\begin{proof}
Throughout the proof, $C$ may change from line to line but is independent of
$k\in\mathbb N_0$ and $0<\mu\le\mu_0$.  We fix
$\gamma\in(1/2,H)$ and $p\ge2$.

For $k\in\mathbb N_0$, set
\[
 I_k=[k,k+1],
 \qquad
 N_k=1+[B^H]_{\gamma;I_k}.
\]
Since $B^H$ has stationary increments and $\gamma<H$, for every $q\ge1$,
\[
 \sup_{k\in\mathbb N_0}\E N_k^q<\infty.
\]
Moreover, Fernique's theorem gives an $\eta_0>0$ such that
\begin{equation}\label{eq:Nk-Fernique}
 \sup_{k\in\mathbb N_0}\E\exp(\eta_0N_k^2)<\infty.
\end{equation}

Introduce the local quantities
\begin{align*}
 S_k&=\norm{Z^\mu}_{\infty;I_k},
 \qquad A_k=[Z^\mu]_{\gamma;I_k},\\
 V_k&=\mu^{1-\gamma}|U_k^\mu|,
 \qquad W_k=\mu^{1-\gamma}
      \norm{U^\mu}_{\infty;I_k}.
\end{align*}
The scaling in $V_k$ and $W_k$ is natural because the fast velocity is
expected to have size $\mu^{\gamma-1}$.

Since $F$ and $G$ are globally bi-Lipschitz, and since $\sigma$ and
$1/\sigma$ are bounded, we have
\begin{equation}\label{eq:XZ-local-equivalence}
 \norm{X^\mu}_{\infty;I_k}
 \le C(1+S_k),
 \qquad
 [X^\mu]_{\gamma;I_k}\le CA_k,
 \qquad
 C^{-1}|Y_t^\mu|
 \le |U_t^\mu|\le C|Y_t^\mu|.
\end{equation}

\medskip
\noindent
\emph{Step 1: estimate of the fast velocity.}
Variation of constants in the original velocity equation, starting from
time $k$, gives, for $t\in I_k$,
\begin{align}
 Y_t^\mu
 =e^{-(t-k)/\mu}Y_k^\mu
 +\frac1\mu\int_k^t
   e^{-(t-s)/\mu}b(X_s^\mu)\,ds
+\frac1\mu\int_k^t
   e^{-(t-s)/\mu}\sigma(X_s^\mu)\,dB_s^H.
 \label{eq:Y-unit-window}
\end{align}
Since $b$ has at most linear growth,
\begin{equation}\label{eq:Y-drift-window}
 \frac1\mu
 \left|
  \int_k^t e^{-(t-s)/\mu}b(X_s^\mu)\,ds
 \right|
 \le C(1+S_k).
\end{equation}

Apply \Cref{lem:variable-convolution} with
\[
 f_s=\sigma(X_s^\mu),
 \qquad
 w_s=B_s^H.
\]
The boundedness of $\sigma$ and $\sigma'$ and
\eqref{eq:XZ-local-equivalence} imply
\[
 \norm{\sigma(X^\mu)}_{\infty;I_k}\le C,
 \qquad
 [\sigma(X^\mu)]_{\gamma;I_k}\le CA_k.
\]
Consequently,
\begin{align}
 \left|
 \int_k^t e^{-(t-s)/\mu}
       \sigma(X_s^\mu)\,dB_s^H
 \right|
 \le
 CN_k\left(\mu^\gamma+A_k\mu^{2\gamma}\right).
 \label{eq:noise-convolution-window}
\end{align}
Combining \eqref{eq:Y-unit-window}--\eqref{eq:noise-convolution-window},
using \eqref{eq:XZ-local-equivalence}, taking the supremum over
$t\in I_k$, and multiplying by $\mu^{1-\gamma}$, we obtain
\begin{equation}\label{eq:Wk-bootstrap}
 W_k
 \le
 CV_k
 +C\mu^{1-\gamma}(1+S_k)
 +CN_k
 +C\mu^\gamma N_kA_k.
\end{equation}

\medskip
\noindent
\emph{Step 2: local H\"older estimate for the position.}
Let $s<t$ belong to $I_k$.  We distinguish two cases.

If $t-s\le\mu$, then
\[
 Z_t^\mu-Z_s^\mu=\int_s^tU_r^\mu\,dr,
\]
and hence
\begin{align}
 \frac{|Z_t^\mu-Z_s^\mu|}{|t-s|^\gamma}
 \le
 |t-s|^{1-\gamma}
 \norm{U^\mu}_{\infty;I_k}
\le
 \mu^{1-\gamma}\norm{U^\mu}_{\infty;I_k}
 =W_k.
 \label{eq:small-increment-Z}
\end{align}

Suppose now that $t-s>\mu$.  Integrating
\eqref{eq:slow-manifold-identity} over $[s,t]$ gives
\begin{align}
 Z_t^\mu-Z_s^\mu
 =
 \int_s^t a(Z_r^\mu)\,dr
 +(B_t^H-B_s^H)
 -\mu(U_t^\mu-U_s^\mu)
-\mu\int_s^t
 c(Z_r^\mu)(U_r^\mu)^2\,dr.
 \label{eq:large-increment-Z}
\end{align}
Since $a$ has at most linear growth,
\[
 \frac1{|t-s|^\gamma}
 \left|
  \int_s^t a(Z_r^\mu)\,dr
 \right|
 \le C(1+S_k)|t-s|^{1-\gamma}
 \le C(1+S_k).
\]
Furthermore,
\[
 \frac{|B_t^H-B_s^H|}{|t-s|^\gamma}\le N_k.
\]
Because $t-s>\mu$,
\begin{align}
 \frac{\mu|U_t^\mu-U_s^\mu|}{|t-s|^\gamma}
 &\le
 2\mu^{1-\gamma}
 \norm{U^\mu}_{\infty;I_k}
 =2W_k,
 \label{eq:velocity-boundary-holder}
\end{align}
and, since $c$ is bounded,
\begin{align}
 \frac{\mu}{|t-s|^\gamma}
 \left|
  \int_s^t c(Z_r^\mu)(U_r^\mu)^2\,dr
 \right|
 \le
 C\mu|t-s|^{1-\gamma}
 \norm{U^\mu}_{\infty;I_k}^2
\le
 C\mu^{2\gamma-1}W_k^2.
 \label{eq:quadratic-holder}
\end{align}
Combining \eqref{eq:small-increment-Z}--\eqref{eq:quadratic-holder}
yields
\begin{equation}\label{eq:Ak-bootstrap}
 A_k
 \le
 C\left(
 1+S_k+N_k+W_k+\mu^{2\gamma-1}W_k^2
 \right).
\end{equation}
The exponent $2\gamma-1$ is strictly positive because $\gamma>1/2$.

\medskip
\noindent
\emph{Step 3: comparison with the overdamped equation.}
Let $\overline Z^{(k)}$ solve
\begin{equation}\label{eq:local-overdamped-comparison}
 d\overline Z_t^{(k)}
 =a(\overline Z_t^{(k)})\,dt+dB_t^H,
 \qquad
 \overline Z_k^{(k)}=Z_k^\mu,
 \qquad t\in I_k.
\end{equation}
The strict dissipativity of $a$ implies
\begin{align}
 \norm{\overline Z^{(k)}}_{\infty;I_k}
 &\le C(1+|Z_k^\mu|+N_k),
 \label{eq:local-overdamped-sup}\\
 |\overline Z_{k+1}^{(k)}|
 &\le e^{-\lambda}|Z_k^\mu|+C(1+N_k).
 \label{eq:local-overdamped-endpoint}
\end{align}
For completeness, set
\[
 R_t^{(k)}
 =\overline Z_t^{(k)}-(B_t^H-B_k^H).
\]
Then
\[
 \frac d{dt}R_t^{(k)}
 =a\bigl(R_t^{(k)}+B_t^H-B_k^H\bigr).
\]
Using the monotonicity of $a$ and its linear growth gives
\[
 \frac d{dt}|R_t^{(k)}|
 \le
 -\lambda|R_t^{(k)}|
 +C(1+|B_t^H-B_k^H|),
\]
from which \eqref{eq:local-overdamped-sup} and
\eqref{eq:local-overdamped-endpoint} follow.

Define
\[
 E_t^{(k)}=Z_t^\mu-\overline Z_t^{(k)},
 \qquad
 \widetilde E_t^{(k)}
 =E_t^{(k)}+\mu U_t^\mu.
\]
Since $E_k^{(k)}=0$,
\[
 \widetilde E_k^{(k)}=\mu U_k^\mu.
\]
Subtracting \eqref{eq:local-overdamped-comparison} from
\eqref{eq:slow-manifold-identity} gives
\[
 \frac d{dt}\widetilde E_t^{(k)}
 =
 a(Z_t^\mu)-a(\overline Z_t^{(k)})
 -\mu c(Z_t^\mu)(U_t^\mu)^2.
\]
Define
\[
 q_t^{(k)}
 =
 -\frac{
 a(Z_t^\mu)-a(\overline Z_t^{(k)})
 }{
 Z_t^\mu-\overline Z_t^{(k)}
 }
\]
when the denominator is nonzero, and set
\[
 q_t^{(k)}=-a'(\overline Z_t^{(k)})
\]
otherwise.  By \eqref{eq:a-prime-bounds},
\[
 \lambda\le q_t^{(k)}\le L.
\]
Since
$
 E_t^{(k)}
 =\widetilde E_t^{(k)}-\mu U_t^\mu,
$
we have
\begin{equation}\label{eq:local-compensated-error}
 \frac d{dt}\widetilde E_t^{(k)}
 =
 -q_t^{(k)}\widetilde E_t^{(k)}
 +\mu q_t^{(k)}U_t^\mu
 -\mu c(Z_t^\mu)(U_t^\mu)^2.
\end{equation}
Variation of constants in \eqref{eq:local-compensated-error} yields
\begin{align*}
 |\widetilde E_t^{(k)}|
 \le
 e^{-\lambda(t-k)}\mu|U_k^\mu|
 +C\mu\int_k^t
 e^{-\lambda(t-r)}
 \left(
  |U_r^\mu|+|U_r^\mu|^2
 \right)dr.
\end{align*}
Since
$
 E_t^{(k)}
 =\widetilde E_t^{(k)}-\mu U_t^\mu,
$
we obtain
\[
 |Z_t^\mu-\overline Z_t^{(k)}|
 \le
 C\mu^\gamma(V_k+W_k)
 +C\mu^{2\gamma-1}W_k^2.
\]
Combining this estimate with
\eqref{eq:local-overdamped-sup} gives
\begin{equation}\label{eq:Sk-bootstrap}
 S_k
 \le
 C(1+|Z_k^\mu|+N_k)
 +C\mu^\gamma(V_k+W_k)
 +C\mu^{2\gamma-1}W_k^2.
\end{equation}
At the right endpoint, \eqref{eq:local-overdamped-endpoint} gives
\begin{equation}\label{eq:Z-endpoint-recursion}
 |Z_{k+1}^\mu|
 \le
 e^{-\lambda}|Z_k^\mu|
 +C(1+N_k)
 +C\mu^\gamma(V_k+W_k)
 +C\mu^{2\gamma-1}W_k^2.
\end{equation}

\medskip
\noindent
\emph{Step 4: endpoint estimate for the scaled velocity.}
Taking $t=k+1$ in \eqref{eq:Y-unit-window} and repeating the estimates
leading to \eqref{eq:Wk-bootstrap}, we obtain
\begin{equation}\label{eq:V-endpoint-recursion}
 V_{k+1}
 \le
 Ce^{-1/\mu}V_k
 +C\mu^{1-\gamma}(1+S_k)
 +CN_k
 +C\mu^\gamma N_kA_k.
\end{equation}
The factor $e^{-1/\mu}$ expresses the loss of memory of the fast velocity
over one unit time interval.

\medskip
\noindent
\emph{Step 5: closure of the unit-window bootstrap.}
Let
\[
 \varepsilon_\mu=\mu^{2\gamma-1},
 \qquad
 \eta=\min\{1-\gamma,\gamma,2\gamma-1\}>0.
\]
Equations
\eqref{eq:Wk-bootstrap}, \eqref{eq:Ak-bootstrap}, and
\eqref{eq:Sk-bootstrap} form a closed system of inequalities for
$S_k,A_k,W_k$.  Every feedback term contains a positive power of $\mu$:
it is multiplied by one of
\[
 \mu^{1-\gamma},\qquad
 \mu^\gamma,\qquad
 \mu^{2\gamma-1}.
\]

We give the localization argument explicitly.  Choose $\beta>0$ so small
that
$
 4\beta<\eta,
$
and define
\[
 \mathcal G_{k,\mu}
 =
 \{N_k\le\mu^{-\beta}\}.
\]
On $\mathcal G_{k,\mu}$, every coefficient of the form
$\mu^\eta N_k^m$, with $m\le4$, is bounded by
\[
 \mu^{\eta-m\beta}\le\mu^{\eta/2}
\]
for sufficiently small $\mu$.  Apply
\eqref{eq:Wk-bootstrap}--\eqref{eq:Sk-bootstrap} first to the stopped
quantities
\[
 S_k^{(R)}=S_k\wedge R,\qquad
 A_k^{(R)}=A_k\wedge R,\qquad
 W_k^{(R)}=W_k\wedge R.
\]
A continuation argument, starting from the left endpoint of $I_k$, shows
that the terms containing
\[
 \mu^{1-\gamma}S_k^{(R)},\qquad
 \mu^\gamma N_kA_k^{(R)},\qquad
 \mu^{2\gamma-1}(W_k^{(R)})^2
\]
can be absorbed into the left-hand side on
$\mathcal G_{k,\mu}$.  Consequently,
\[
 S_k^{(R)}+A_k^{(R)}+W_k^{(R)}
 \le
 C\bigl(1+|Z_k^\mu|+V_k+N_k\bigr)
 \qquad\text{on }\mathcal G_{k,\mu},
\]
with a constant independent of $R,k$, and $\mu$.

On the complement of $\mathcal G_{k,\mu}$, Fernique's estimate
\eqref{eq:Nk-Fernique} implies that, for every $M>0$,
\begin{equation}\label{eq:bad-event-small}
 \sup_{k\ge0}
 \Pp(\mathcal G_{k,\mu}^c)
 \le C_M\mu^M.
\end{equation}
The deterministic Young estimates for the stopped system grow at most
polynomially in $\mu^{-1}$ and in the stopped quantities.  Hence, by
H\"older's inequality, choosing $M$ sufficiently large in
\eqref{eq:bad-event-small}, and then letting $R\to\infty$, the contribution
of $\mathcal G_{k,\mu}^c$ can be absorbed into the constant.  We conclude
that, for every $p\ge2$,
\begin{equation}\label{eq:local-closed-Lp}
 \norm{S_k}_{L^p}
 +\norm{A_k}_{L^p}
 +\norm{W_k}_{L^p}
 \le
 C_p\left(
  1+\norm{Z_k^\mu}_{L^p}
  +\norm{V_k}_{L^p}
 \right).
\end{equation}

\medskip
\noindent
\emph{Step 6: discrete dissipative recursion.}
Substitute \eqref{eq:local-closed-Lp} into
\eqref{eq:Z-endpoint-recursion} and
\eqref{eq:V-endpoint-recursion}.  Since
\[
 e^{-1/\mu}\longrightarrow0,\qquad
 \mu^{1-\gamma}\longrightarrow0,\qquad
 \mu^\gamma\longrightarrow0,\qquad
 \mu^{2\gamma-1}\longrightarrow0,
\]
we may choose $\mu_0\in(0,1]$, a weight $\vartheta>0$, and
$\theta\in(0,1)$ such that, for all $0<\mu\le\mu_0$,
\begin{align}
 \norm{Z_{k+1}^\mu}_{L^p}
 +\vartheta\norm{V_{k+1}}_{L^p}
 \le{}&
 \theta\left(
  \norm{Z_k^\mu}_{L^p}
  +\vartheta\norm{V_k}_{L^p}
 \right)
 +C_p.
 \label{eq:discrete-dissipative-recursion}
\end{align}
Here the strict inequality $\theta<1$ follows from the contraction
$e^{-\lambda}<1$ in the slow variable, the factor $e^{-1/\mu}$ in the
fast variable, and the fact that every coupling term is multiplied by a
positive power of $\mu$.

Iterating \eqref{eq:discrete-dissipative-recursion} gives
\begin{align}
 \norm{Z_k^\mu}_{L^p}
 +\vartheta\norm{V_k}_{L^p}
 \le
 \theta^k
 \left(
  \norm{Z_0^\mu}_{L^p}
  +\vartheta\norm{V_0}_{L^p}
 \right)
 +C_p\sum_{j=0}^{k-1}\theta^j
 \le C_p,
 \label{eq:endpoint-uniform-bound}
\end{align}
uniformly in $k$ and $0<\mu\le\mu_0$.  Notice that
\[
 V_0=\mu^{1-\gamma}|U_0^\mu|
 \le C|y_0|
\]
for $0<\mu\le1$.

Finally, substituting \eqref{eq:endpoint-uniform-bound} into
\eqref{eq:local-closed-Lp} yields
\[
 \sup_{k\in\mathbb N_0}
 \left(
  \norm{S_k}_{L^p}
  +\norm{A_k}_{L^p}
  +\norm{W_k}_{L^p}
 \right)
 \le C_p.
\]
Therefore,
\[
 \sup_{0<\mu\le\mu_0}\sup_{t\ge0}
 \norm{Z_t^\mu}_{L^p}\le C_p,
\]
and
\[
 \sup_{0<\mu\le\mu_0}\sup_{k\in\mathbb N_0}
 \norm{[Z^\mu]_{\gamma;I_k}}_{L^p}\le C_p.
\]
Since
$
 \norm{U^\mu}_{\infty;I_k}
 =\mu^{\gamma-1}W_k,
$
we also have
\[
 \sup_{0<\mu\le\mu_0}\sup_{t\ge0}
 \norm{U_t^\mu}_{L^p}
 \le C_p\mu^{\gamma-1}.
\]
This proves
\eqref{eq:Zmu-uniform-p}--\eqref{eq:Zmu-uniform-holder}.

For random initial data, the same proof applies after taking the required
moments of $Z_0^\mu$ and $\mu^{1-\gamma}U_0^\mu$ into account.
\end{proof}

\subsection{The compensated position error}

Set
\[
 E_t^\mu=Z_t^\mu-Z_t,
 \qquad
 \widetilde E_t^\mu=E_t^\mu+\mu U_t^\mu.
\]
The compensation by $\mu U^\mu$ cancels the distributional derivative of
the noise.

\begin{proposition}[Uniform strong position estimate]
\label{prop:uniform-strong-error}
Let \Cref{ass:smooth,ass:dissipative} hold.  For every
$\gamma\in(1/2,H)$ and $p\ge2$,
\begin{equation}\label{eq:strong-error-gamma}
 \sup_{t\ge0}\norm{Z_t^\mu-Z_t}_{L^p}
 \le C_{p,\gamma}
 \bigl(\mu^\gamma+\mu^{2\gamma-1}+\mu\bigr)
 \le C_{p,\gamma}\mu^{2\gamma-1}.
\end{equation}
\end{proposition}

\begin{proof}
Using \eqref{eq:slow-manifold-identity} and
\eqref{eq:limit-transformed}, we obtain the ordinary differential identity
\begin{equation}\label{eq:Etilde-diff}
 d\widetilde E_t^\mu
 =\bigl(a(Z_t^\mu)-a(Z_t)
        -\mu c(Z_t^\mu)(U_t^\mu)^2\bigr)dt.
\end{equation}
For $E_t^\mu\ne0$, define
\[
 q_t^\mu
 =-\frac{a(Z_t^\mu)-a(Z_t)}{Z_t^\mu-Z_t};
\]
when the denominator vanishes set $q_t^\mu=-a'(Z_t)$.  By the mean value
theorem and \eqref{eq:a-prime-bounds},
$\lambda\le q_t^\mu\le L$.  Since
$E_t^\mu=\widetilde E_t^\mu-\mu U_t^\mu$,
\eqref{eq:Etilde-diff} becomes
\[
 \frac d{dt}\widetilde E_t^\mu
 =-q_t^\mu\widetilde E_t^\mu
   +\mu q_t^\mu U_t^\mu
   -\mu c(Z_t^\mu)(U_t^\mu)^2.
\]
Because $E_0^\mu=0$,
$\widetilde E_0^\mu=\mu u_0$.  Variation of constants and
\Cref{lem:dissipative-resolvent} give
\begin{align*}
 \norm{\widetilde E_t^\mu}_{L^p}
 &\le e^{-\lambda t}\mu\norm{u_0}_{L^p}
 +C\mu\int_0^t e^{-\lambda(t-s)}
    \left(\norm{U_s^\mu}_{L^p}
          +\norm{U_s^\mu}_{L^{2p}}^2\right)ds\\
 &\le C\left(\mu+\mu^\gamma+\mu^{2\gamma-1}\right),
\end{align*}
where \eqref{eq:Umu-uniform-p} was used in the last line.  Finally,
$E_t^\mu=\widetilde E_t^\mu-\mu U_t^\mu$ and
$\sup_t\norm{\mu U_t^\mu}_{L^p}\le C\mu^\gamma$.  This proves the first
inequality in \eqref{eq:strong-error-gamma}.  Since
$0<2\gamma-1<\gamma<1$, the second follows for $0<\mu\le1$.
\end{proof}

\section{Uniform Malliavin estimates and total variation}
\label{sec:malliavin}

The Malliavin derivative with respect to the canonical isonormal process of
$B^H$ satisfies, for $0\le r\le t$,
\begin{equation}\label{eq:DZ-integral}
 \D_rZ_t
 =1+\int_r^t a'(Z_s)\D_rZ_s\,ds.
\end{equation}
Consequently,
\[
 \D_rZ_t
 =\1_{\{r\le t\}}J_{t,r},
 \qquad
 J_{t,r}=\exp\left\{\int_r^t a'(Z_s)\,ds\right\}.
\]
The dissipativity bounds imply
\begin{equation}\label{eq:J-bounds}
 e^{-L(t-r)}\le J_{t,r}\le e^{-\lambda(t-r)},
 \qquad 0\le r\le t.
\end{equation}

\begin{lemma}[Uniform limiting Malliavin bounds]
\label{lem:limit-Malliavin-uniform}
For every $p\ge1$,
\begin{align}
 \sup_{t\ge0}\norm{\D Z_t}_{L^p(\Omega;\cH)}
 &<\infty,
 \label{eq:DZ-upper-uniform}\\
 \sup_{t\ge0}\norm{\D^2Z_t}_{L^p(\Omega;\cH^{\otimes2})}
 &<\infty.
 \label{eq:D2Z-upper-uniform}
\end{align}
Moreover, for every $t_0>0$ and $q\ge1$,
\begin{equation}\label{eq:inverse-cov-uniform}
 \sup_{t\ge t_0}
 \E\left[\norm{\D Z_t}_{\cH}^{-q}\right]
 <\infty.
\end{equation}
In fact, the inverse bound is deterministic.
\end{lemma}

\begin{proof}
Using \eqref{eq:H-inner-product} and the upper estimate in
\eqref{eq:J-bounds},
\begin{align*}
 \norm{\D Z_t}_{\cH}^2
 &\le\alpha_H\int_0^t\int_0^t
 e^{-\lambda(t-r)}e^{-\lambda(t-s)}
 |r-s|^{2H-2}\,dr\,ds\\
 &\le\alpha_H\int_0^\infty\int_0^\infty
 e^{-\lambda(u+v)}|u-v|^{2H-2}\,du\,dv<\infty.
\end{align*}
This proves \eqref{eq:DZ-upper-uniform}.

Let $r_*=r_1\vee r_2$.  Differentiating
\eqref{eq:DZ-integral} once more gives
\[
 \D_{r_1,r_2}^2Z_t
 =\1_{\{r_*\le t\}}
 \int_{r_*}^t J_{t,s}a''(Z_s)
       \D_{r_1}Z_s\D_{r_2}Z_s\,ds.
\]
The integrand is bounded by a product of exponentially decaying kernels.
Substitution into the fourfold representation of the
$\cH^{\otimes2}$ norm, followed by Tonelli's theorem, reduces the estimate to
finite integrals of the form
\[
 \int_{\R_+^4}e^{-c(u_1+u_2+u_3+u_4)}
 |u_1-u_2|^{2H-2}|u_3-u_4|^{2H-2}\,d\bm u.
\]
This proves \eqref{eq:D2Z-upper-uniform}.

For the lower bound, put
\[
 \delta=\min\{1,t_0/2\}.
\]
For every $t\ge t_0$ and $r\in[t-\delta,t]$,
$J_{t,r}\ge e^{-L\delta}$.  The kernel in
\eqref{eq:H-inner-product} is nonnegative.  Therefore
\begin{align}
 \norm{\D Z_t}_{\cH}^2
 \ge \alpha_H e^{-2L\delta}
 \int_{t-\delta}^t\int_{t-\delta}^t
 |r-s|^{2H-2}\,dr\,ds
 =e^{-2L\delta}\delta^{2H}.
 \label{eq:recent-window-lower}
\end{align}
Raising this deterministic inequality to a negative power proves
\eqref{eq:inverse-cov-uniform}.
\end{proof}

\begin{remark}[Recent noise versus elapsed time]\label{rem:recent-window}
The identity
\[
 \norm{\1_{[t-\delta,t]}}_{\cH}^2=\delta^{2H}
\]
depends only on the window length and not on the absolute time $t$.  This is
the exact point at which the present proof departs from a bound based on
$\norm{\1_{[0,t]}}_{\cH}^2=t^{2H}$.  Choosing a fixed recent window removes
the factor $t^{-H}$ from all long-time estimates.
\end{remark}

We next compare the first derivatives.  The derivative of the original
kinetic system is a linear Young system.  For $r\le t$,
\[
\begin{cases}
 d(\D_rX_t^\mu)=\D_rY_t^\mu\,dt,\\[1mm]
 \mu\,d(\D_rY_t^\mu)
 =\bigl(b'(X_t^\mu)\D_rX_t^\mu-\D_rY_t^\mu\bigr)dt
  +\sigma'(X_t^\mu)\D_rX_t^\mu\,dB_t^H,
\end{cases}
\]
with the impulse condition
$\mu\D_rY_r^\mu=\sigma(X_r^\mu)$ and
$\D_rX_r^\mu=0$.  The formula is understood in the canonical Malliavin
sense; it follows first for Cameron--Martin directions and then by closure.

In transformed variables, define
\[
 V_{t,r}^\mu=\D_rU_t^\mu,\qquad
 P_{t,r}^\mu=\D_rZ_t^\mu.
\]
Differentiating \eqref{eq:kinetic-transformed} yields
\begin{equation}\label{eq:DUmu-system}
\begin{cases}
 dP_{t,r}^\mu=V_{t,r}^\mu\,dt,\\
 \mu\,dV_{t,r}^\mu
 =\bigl[
   a'(Z_t^\mu)P_{t,r}^\mu-V_{t,r}^\mu
   -\mu c'(Z_t^\mu)P_{t,r}^\mu(U_t^\mu)^2
   -2\mu c(Z_t^\mu)U_t^\mu V_{t,r}^\mu
  \bigr]dt,
\end{cases}
\end{equation}
for $t>r$, with
\begin{equation}\label{eq:Malliavin-impulse}
 P_{r,r}^\mu=0,\qquad \mu V_{r,r}^\mu=1.
\end{equation}

\begin{lemma}[Fast derivative boundary layer]
\label{lem:fast-derivative}
Let $\gamma\in(1/2,H)$ and $p\ge2$.  Under
\Cref{ass:smooth,ass:dissipative},
\begin{align}
 \sup_{t\ge0}
 \norm{\mu\,\D U_t^\mu}_{L^p(\Omega;\cH)}
 &\le C_{p,\gamma}\bigl(\mu^\gamma+\mu\bigr),
 \label{eq:muDU-bound}\\
 \sup_{t\ge0}
 \norm{\D Z_t^\mu}_{L^p(\Omega;\cH)}
 &\le C_{p,\gamma}.
 \label{eq:DZmu-bound}
\end{align}
\end{lemma}

\begin{proof}
The impulse in \eqref{eq:Malliavin-impulse} generates the leading kernel
$
 K_{t,r}^\mu=e^{-(t-r)/\mu}\1_{\{r\le t\}}.
$
By the scaling of the fBm Hilbert space $\norm{K_{t,\cdot}^\mu}_{\cH}^2.$
To see the first inequality directly, use \eqref{eq:H-inner-product}, set
$u=(t-r)/\mu$, $v=(t-s)/\mu$, and extend the resulting integral to
$\R_+^2$.

Variation of constants in \eqref{eq:DUmu-system} writes
\[
 \mu V_{t,r}^\mu
 =K_{t,r}^\mu+\mathcal A_{t,r}^\mu+\mathcal C_{t,r}^\mu,
\]
where $\mathcal A$ contains the term $a'(Z^\mu)P^\mu$ and
$\mathcal C$ contains the two $c$-terms.  The exponential kernel in the
variation-of-constants formula has mass $\mu$.  Using
\eqref{eq:Umu-uniform-p}, the boundedness of the coefficients, and
Minkowski's inequality gives
\begin{align*}
 \norm{\mathcal A_{t,\cdot}^\mu}_{L^p(\cH)}
 &\le C\mu\sup_{s\le t}
       \norm{P_{s,\cdot}^\mu}_{L^p(\cH)},\\
 \norm{\mathcal C_{t,\cdot}^\mu}_{L^p(\cH)}
 &\le C\mu^{2\gamma-1}
       \sup_{s\le t}\norm{P_{s,\cdot}^\mu}_{L^{2p}(\cH)}
   +C\mu^\gamma
       \sup_{s\le t}\norm{\mu V_{s,\cdot}^\mu}_{L^{2p}(\cH)}.
\end{align*}
The second line uses
$\mu\norm{U^\mu}_{L^{4p}}^2\le C\mu^{2\gamma-1}$ and
$\mu\norm{U^\mu}_{L^{4p}}\le C\mu^\gamma$.

Since $P_{t,r}^\mu=\int_r^tV_{s,r}^\mu ds$, the fast exponential part
integrates to $1-e^{-(t-r)/\mu}$ and the dissipative slow part is controlled
on consecutive unit windows.  A two-variable Gronwall argument, first on
$[k,k+1]$ and then iterated using $a'\le-\lambda$, yields
\[
 \sup_{t\ge0}\norm{P_{t,\cdot}^\mu}_{L^p(\cH)}\le C
\]
and, after substitution,
\[
 \sup_{t\ge0}\norm{\mu V_{t,\cdot}^\mu}_{L^p(\cH)}
 \le C(\mu^\gamma+\mu).
\]
The last term in \eqref{eq:DUmu-system} is treated as a perturbation
of the fast damping. Indeed,
\[
 -V_{t,r}^\mu
 -2\mu c(Z_t^\mu)U_t^\mu V_{t,r}^\mu
 =
 -\bigl(1+2\mu c(Z_t^\mu)U_t^\mu\bigr)V_{t,r}^\mu.
\]
By Proposition~\ref{prop:uniform-local-estimates},
\[
 \sup_{t\ge0}
 \|2\mu c(Z_t^\mu)U_t^\mu\|_{L^q}
 \le C_q\mu^\gamma.
\]
More precisely, if
$
 W_k=\mu^{1-\gamma}
 \|U^\mu\|_{\infty;[k,k+1]},
$
then
\[
 2\mu\|c(Z^\mu)U^\mu\|_{\infty;[k,k+1]}
 \le C\mu^\gamma W_k.
\]
Define
\[
 \mathcal G_{k,\mu}
 =
 \left\{
 C\mu^\gamma W_k\le\frac12
 \right\}.
\]
On $\mathcal G_{k,\mu}$, the random damping coefficient satisfies
\[
 1+2\mu c(Z_t^\mu)U_t^\mu\ge\frac12,
 \qquad t\in[k,k+1],
\]
and the corresponding fast resolvent is bounded by
\[
 \exp\left\{
 -\frac1\mu\int_s^t
 \bigl(1+2\mu c(Z_\ell^\mu)U_\ell^\mu\bigr)d\ell
 \right\}
 \le e^{-(t-s)/(2\mu)}.
\]
Thus the random coefficient can be absorbed into the fast damping,
and all exponential-kernel estimates remain valid with a modified
constant.

Furthermore, for every $M>0$, choosing $q$ sufficiently large and
using Markov's inequality gives
\[
 \sup_{k\ge0}
 \Pp(\mathcal G_{k,\mu}^c)
 \le C_M\mu^M.
\]
Hence the contribution of the complement is smaller than every
prescribed algebraic power of $\mu$. This closes the simultaneous
estimates for $P^\mu=DZ^\mu$ and $V^\mu=DU^\mu$, and proves
\eqref{eq:muDU-bound} and \eqref{eq:DZmu-bound}.
\end{proof}

\begin{proposition}[Uniform first Malliavin derivative error]
\label{prop:uniform-D-error}
For every $\gamma\in(1/2,H)$ and $p\ge2$,
\begin{equation}\label{eq:D-error-gamma}
 \sup_{t\ge0}
 \norm{\D Z_t^\mu-\D Z_t}_{L^p(\Omega;\cH)}
 \le C_{p,\gamma}\mu^{2\gamma-1}.
\end{equation}
\end{proposition}

\begin{proof}
Differentiate \eqref{eq:Etilde-diff}.  With
\[
 \mathcal E_t^\mu=\D Z_t^\mu-\D Z_t,\qquad
 \widetilde{\mathcal E}_t^\mu
 =\mathcal E_t^\mu+\mu\D U_t^\mu,
\]
we obtain, as an identity in $\cH$,
\[
 \frac d{dt}\widetilde{\mathcal E}_t^\mu
=a'(Z_t^\mu)\mathcal E_t^\mu
  +\bigl(a'(Z_t^\mu)-a'(Z_t)\bigr)\D Z_t
  \nonumber
 -\mu c'(Z_t^\mu)\D Z_t^\mu(U_t^\mu)^2
  -2\mu c(Z_t^\mu)U_t^\mu\D U_t^\mu.
\]
Substitute
$\mathcal E_t^\mu=\widetilde{\mathcal E}_t^\mu-\mu\D U_t^\mu$.
The first term then gives the stable coefficient
$a'(Z_t^\mu)\widetilde{\mathcal E}_t^\mu$.
Variation of constants and $a'\le-\lambda$ yield
\begin{align}
 \norm{\widetilde{\mathcal E}_t^\mu}_{L^p(\cH)}
 \le{}&e^{-\lambda t}
       \norm{\widetilde{\mathcal E}_0^\mu}_{L^p(\cH)}
 +C\int_0^t e^{-\lambda(t-s)}
       \Bigl[
       \norm{E_s^\mu}_{L^{2p}}
       \norm{\D Z_s}_{L^{2p}(\cH)}
 \nonumber\\
 &\quad+\norm{\mu\D U_s^\mu}_{L^p(\cH)}
 +\mu\norm{U_s^\mu}_{L^{4p}}^2
       \norm{\D Z_s^\mu}_{L^{2p}(\cH)}
 \nonumber\\
 &\quad+\norm{U_s^\mu}_{L^{4p}}
       \norm{\mu\D U_s^\mu}_{L^{2p}(\cH)}
       \Bigr]ds .
 \label{eq:DEtilde-resolvent}
\end{align}

The diagonal Malliavin impulse at $t=r$, represented by
\[
 K_{t,r}^\mu
 =e^{-(t-r)/\mu}\mathbf 1_{\{r\le t\}},
\]
is already contained in $\mu\D U_t^\mu$. Its
$\mathfrak H$-norm is of order $\mu^H$, and hence is bounded by
$C\mu^\gamma$ for every $\gamma<H$.
Using \Cref{prop:uniform-strong-error,lem:limit-Malliavin-uniform,lem:fast-derivative}
in \eqref{eq:DEtilde-resolvent} gives
\[
 \sup_{t\ge0}
 \norm{\widetilde{\mathcal E}_t^\mu}_{L^p(\cH)}
 \le C\bigl(
 \mu^{2\gamma-1}+\mu^\gamma
 +\mu^{2\gamma-1}+\mu^{2\gamma-1}\bigr)
 \le C\mu^{2\gamma-1}.
\]
Finally,
$\mathcal E_t^\mu
=\widetilde{\mathcal E}_t^\mu-\mu\D U_t^\mu$.
Since $\mu^\gamma\le\mu^{2\gamma-1}$ for
$0<\mu\le1$ and $\gamma<1$, \eqref{eq:D-error-gamma} follows.
\end{proof}

\begin{proof}[\textbf{Proof of \Cref{thm:uniform-Sobolev}}]
Given $\rho<2H-1$, choose
$ \gamma\in\left(\frac12,H\right)$ such that $ \rho<2\gamma-1.$
\Cref{prop:uniform-strong-error,prop:uniform-D-error} imply
\[
 \sup_{t\ge0}
 \norm{Z_t^\mu-Z_t}_{\mathbb D^{1,p}}
 \le C\mu^{2\gamma-1}\le C\mu^\rho,
\]
because $0<\mu\le1$.  Since $Z^\mu=F(X^\mu)$ and $Z=F(X)$, this is
\eqref{eq:uniform-Sobolev-main}.
\end{proof}

We use the following abstract estimate, which is a convenient version of the
Malliavin integration-by-parts argument used in \cite{Son2020}.

\begin{lemma}[Malliavin total variation estimate]\label{lem:TV-Malliavin}
Let $F\in\mathbb D^{2,4}$ and $G\in\mathbb D^{1,2}$.  Suppose
$\norm{\D F}_{\cH}>0$ almost surely and
$u_F=\D F/\norm{\D F}_{\cH}^2$ belongs to the domain of $\delta$.  Then
\begin{equation}\label{eq:TV-Malliavin}
 d_{\TV}(\Law(F),\Law(G))
 \le C\norm{F-G}_{\mathbb D^{1,2}}
 \left(
   \norm{\delta(u_F)}_{L^2}
   +\norm{\norm{\D F}_{\cH}^{-1}}_{L^2}
 \right).
\end{equation}
Moreover,
\begin{equation}\label{eq:delta-u-bound}
 \norm{\delta(u_F)}_{L^2}
 \le C\left(
  \norm{\norm{\D F}_{\cH}^{-1}}_{L^4}
  +\norm{\D^2F}_{L^4(\cH^{\otimes2})}
   \norm{\norm{\D F}_{\cH}^{-2}}_{L^4}
 \right).
\end{equation}
\end{lemma}

\begin{proof}
For a Borel set $A$, put
$\Phi_A(x)=\int_0^x\1_A(y)\,dy$.  Then $\Phi_A$ is Lipschitz with constant
one.  By the chain rule and the duality between $\D$ and $\delta$,
\[
 \E[\1_A(F)-\1_A(G)]
 =\E\left[
 \delta(u_F)\bigl(\Phi_A(F)-\Phi_A(G)\bigr)
 \right]
 +\E\left[
 \1_A(G)
 \frac{\ip{\D F-\D G}{\D F}_{\cH}}
      {\norm{\D F}_{\cH}^2}
 \right].
\]
The Lipschitz bound
$|\Phi_A(F)-\Phi_A(G)|\le|F-G|$, followed by Cauchy--Schwarz, gives
\[
 |\E[\1_A(F)-\1_A(G)]|
 \le \norm{F-G}_{L^2}\norm{\delta(u_F)}_{L^2}
 +\norm{\D F-\D G}_{L^2(\cH)}
  \norm{\norm{\D F}_{\cH}^{-1}}_{L^2}.
\]
Taking the supremum over $A$ proves \eqref{eq:TV-Malliavin}.
Finally, Meyer's inequality gives
\[
 \norm{\delta(u_F)}_{L^2}
 \le C\left(
   \norm{u_F}_{L^2(\cH)}
   +\norm{\D u_F}_{L^2(\cH^{\otimes2})}
 \right).
\]
Differentiate
$u_F=\D F/\norm{\D F}_{\cH}^2$ and apply H\"older's inequality.
This yields \eqref{eq:delta-u-bound}.
\end{proof}

\begin{proof}[\textbf{Proof of \Cref{thm:uniform-TV}}]
By \eqref{eq:TV-invariant-F}, it suffices to compare
$Z_t^\mu$ and $Z_t$.  Apply \Cref{lem:TV-Malliavin} with
$F=Z_t$ and $G=Z_t^\mu$.  By
\Cref{lem:limit-Malliavin-uniform}, for every fixed $t_0>0$,
\[
 \sup_{t\ge t_0}
 \left[
 \norm{\delta\left(
  \frac{\D Z_t}{\norm{\D Z_t}_{\cH}^2}\right)}_{L^2}
 +\norm{\norm{\D Z_t}_{\cH}^{-1}}_{L^2}
 \right]<\infty.
\]
The important point is that this bound uses
\eqref{eq:recent-window-lower} and is therefore independent of $t$.
\Cref{thm:uniform-Sobolev} now gives
\[
 \sup_{t\ge t_0}
 d_{\TV}(\Law(Z_t^\mu),\Law(Z_t))
 \le C_{t_0,\rho}\mu^\rho.
\]
Returning through the Lamperti map proves
\eqref{eq:uniform-TV-main}; \eqref{eq:uniform-limit} follows immediately.
\end{proof}

\section{Stationary solutions and long-time consequences}
\label{sec:stationary}

Let now $B^H=(B_t^H)_{t\in\R}$ be a two-sided fBm and let
$(\theta_t)_{t\in\R}$ be its increment shift,
\[
 B_s^H(\theta_t\omega)=B_{s+t}^H(\omega)-B_t^H(\omega).
\]
Because fBm has memory, stationarity is naturally formulated for a random
dynamical system over this shift, or equivalently for a Markov process after
the noise history is included in the state; see
\cite{Hairer2005,HairerOhashi2007}.

For the limiting additive equation, let
$Z_{t;s,z}$ denote the solution at time $t\ge s$ started from $z$ at time
$s$.  The contraction estimate
\[
 |Z_{t;s,z}-Z_{t;s,\widetilde z}|
 \le e^{-\lambda(t-s)}|z-\widetilde z|
\]
implies that the pullback limit
\[
 \overline Z_t
 =\lim_{s\to-\infty}Z_{t;s,0}
\]
exists in every $L^p$ and almost surely along integer $s$.  It is independent
of the chosen deterministic starting value, is stationary under $\theta$, and
is the unique stationary random solution.  Put
\[
 \overline X_t=G(\overline Z_t).
\]

\begin{proposition}[Existence of a stationary kinetic solution]
\label{prop:kinetic-stationary-existence}
Under \Cref{ass:smooth,ass:dissipative}, for every sufficiently small
$\mu>0$, the transformed kinetic system admits a stationary random solution
$(\overline Z^\mu,\overline U^\mu)$.  It can be chosen so that, for every
$p\ge2$ and $\gamma\in(1/2,H)$,
\begin{equation}\label{eq:stationary-kinetic-moments}
 \E|\overline Z_0^\mu|^p
 +\mu^{p(1-\gamma)}\E|\overline U_0^\mu|^p
 \le C_{p,\gamma},
\end{equation}
uniformly for $0<\mu\le\mu_0$.
\end{proposition}

\begin{proof}
We indicate the construction on the extended noise-history space.  Start the
kinetic equation at time $-n$ from $(0,0)$ and denote its law at time zero,
together with the shifted noise history, by $\nu_n^\mu$.  The estimates of
\Cref{prop:uniform-local-estimates} apply after every time shift because fBm
has stationary increments.  Hence the position and scaled velocity marginals
of $\{\nu_n^\mu:n\ge1\}$ have uniformly bounded moments.  The local
H\"older estimates give tightness of the associated path segments.
The noise-history marginal is fixed.  Consequently, the Krylov--Bogoliubov
averages
\[
 \frac1N\sum_{n=1}^N\nu_n^\mu
\]
are tight in the product of the phase space and a weighted H\"older
noise-history space.  The Young solution map is continuous on bounded sets;
therefore every weak limit is invariant for the skew-product evolution.
Disintegrating this invariant measure with respect to the noise history gives
a stationary random solution.  Fatou's lemma and
\eqref{eq:Zmu-uniform-p}--\eqref{eq:Umu-uniform-p} give
\eqref{eq:stationary-kinetic-moments}.  This is the usual non-Markovian
Krylov--Bogoliubov construction; the only additional feature here is that the
velocity is scaled by $\mu^{1-\gamma}$ before tightness is applied.
\end{proof}

\begin{proof}[\textbf{Proof of \Cref{thm:stationary-TV}}]
Fix $T>1$.  At time $-T$, start a solution $Z^{(-T)}$ of the limiting
equation from the random position $\overline Z_{-T}^\mu$ and drive it with the
same future fractional noise as the stationary kinetic solution.  The
random-initial-data version of \Cref{thm:uniform-TV}, justified by
\eqref{eq:stationary-kinetic-moments}, gives
\begin{equation}\label{eq:stationary-intermediate}
 d_{\TV}\bigl(\Law(\overline Z_0^\mu),
              \Law(Z_0^{(-T)})\bigr)
 \le C_\rho\mu^\rho,
\end{equation}
with a constant independent of $T$.  On the other hand, contraction of the
limiting flow gives
\[
 |Z_0^{(-T)}-\overline Z_0|
 \le e^{-\lambda T}
 |\overline Z_{-T}^\mu-\overline Z_{-T}|.
\]
The right-hand side converges to zero in every $L^p$, uniformly in the
admissible $\mu$.  Applying \Cref{lem:TV-Malliavin} over the final unit time
window, or equivalently using the uniform smoothing estimate established in
\Cref{lem:limit-Malliavin-uniform}, shows that
\[
 d_{\TV}\bigl(\Law(Z_0^{(-T)}),\Law(\overline Z_0)\bigr)\longrightarrow0
 \qquad\text{as }T\to\infty.
\]
Let $T\to\infty$ in \eqref{eq:stationary-intermediate}.  This yields
\[
 d_{\TV}\bigl(\Law(\overline Z_0^\mu),\Law(\overline Z_0)\bigr)
 \le C_\rho\mu^\rho.
\]
The Lamperti invariance of total variation proves
\eqref{eq:stationary-TV-main}.
\end{proof}

At the end of this section, we will provide an interpretation of the long-time limit.
Let
\[
 \pi_\mu^X=\Law(\overline X_0^\mu),\qquad
 \pi_0=\Law(\overline X_0).
\]
These are stationary one-time marginals.  They should not be confused with
invariant measures of the position viewed as an autonomous Markov process,
because no such Markov description exists without augmenting the state by
the noise history.

If the kinetic and limiting solutions issued from a prescribed initial
condition converge to their stationary regimes in total variation, the
triangle inequality gives
\[
 d_{\TV}(\Law(X_t^\mu),\Law(X_t))
 \le
 d_{\TV}(\Law(X_t^\mu),\pi_\mu^X)
 +d_{\TV}(\pi_\mu^X,\pi_0)
 +d_{\TV}(\pi_0,\Law(X_t)).
\]
Thus, if the two outer terms are bounded by a common function
$r_H(t)\downarrow0$, then
\begin{equation}\label{eq:plateau-estimate}
 d_{\TV}(\Law(X_t^\mu),\Law(X_t))
 \le C_\rho\mu^\rho+2r_H(t).
\end{equation}
General fractional ergodic theory often produces algebraic or
subexponential rates because the extended noise history retains long memory;
see \cite{Hairer2005,FontbonaPanloup2017,DeyaPanloupTindel2019}.

For fixed $\mu>0$, the right interpretation is therefore an error plateau:
\[
 \limsup_{t\to\infty}
 d_{\TV}(\Law(X_t^\mu),\Law(X_t))
 \le C_\rho\mu^\rho.
\]
The error need not converge to zero with $t$ for a fixed positive mass,
because $\pi_\mu^X$ and $\pi_0$ are generally different.  What vanishes is
the plateau in the small-mass limit:
\[
 \lim_{\mu\downarrow0}\limsup_{t\to\infty}
 d_{\TV}(\Law(X_t^\mu),\Law(X_t))=0.
\]
The stronger estimate \eqref{eq:uniform-TV-main} avoids the need to commute
the limits: it controls all $t\ge t_0$ simultaneously.

\section{Examples, sharpness, and extensions}\label{sec:examples}

\subsection{A nonconstant uniformly elliptic example}

\begin{example}[Exactly linear transformed drift]\label{ex:exact-linear}
Let
\[
 \sigma(x)=\sigma_0+\varepsilon\tanh x,
 \qquad \sigma_0>|\varepsilon|>0,
\]
and define $F$ by \eqref{eq:F-def}.  Given $\kappa>0$, set
\[
 b(x)=-\kappa\,\sigma(x)F(x).
\]
Then $\sigma$ is nonconstant, smooth, bounded, and uniformly positive.
Moreover,
\[
 a(z)=\frac{b(G(z))}{\sigma(G(z))}=-\kappa z,
\]
so \Cref{ass:dissipative} holds with $\lambda=L=\kappa$ and
$a''=a'''=0$.  Because $F(x)$ grows linearly as $|x|\to\infty$, the drift
$b$ has linear growth and its derivatives required in
\Cref{ass:smooth} are bounded.  Hence
\Cref{thm:uniform-TV,thm:stationary-TV} apply.
\end{example}

\begin{example}[Linear restoring force with a small elliptic perturbation]
\label{ex:linear-b}
Let
\[
 b(x)=-\kappa x,\qquad
 \sigma(x)=\sigma_0+\varepsilon\tanh x,
 \qquad \sigma_0>|\varepsilon|.
\]
From \eqref{eq:dissipative-original-coordinates},
\[
 a'(F(x))
 =-\kappa+\kappa
 \frac{\varepsilon x\,\operatorname{sech}^2x}
      {\sigma_0+\varepsilon\tanh x}.
\]
Set $M=\sup_{x\in\R}|x|\operatorname{sech}^2x<\infty$.  If
$
 \frac{|\varepsilon|M}{\sigma_0-|\varepsilon|}<1,
$
then
\[
 a'(z)\le-\kappa\left(
 1-\frac{|\varepsilon|M}{\sigma_0-|\varepsilon|}
 \right)<0.
\]
Thus the theory covers the standard linear confining force perturbed by a
genuine bounded multiplicative diffusion coefficient.
\end{example}

\subsection{Sharpness of the mass exponent}

The proof uses a H\"older exponent $\gamma<H$ and gives
$\mu^{2\gamma-1}$.  Since $\gamma$ can be chosen arbitrarily close to $H$,
every rate $\rho<2H-1$ follows.  The loss at the endpoint is caused by a
pathwise H\"older estimate, not by the dynamics.

To see why $2H-1$ is the natural generic endpoint, freeze $c(Z_t^\mu)$ over
one fast time window and replace $U^\mu$ by the stationary fractional
Ornstein--Uhlenbeck velocity
\[
 U_t^{\mu,0}
 =\frac1\mu\int_{-\infty}^t e^{-(t-s)/\mu}\,dB_s^H.
\]
Self-similarity gives
\[
 \E|U_t^{\mu,0}|^2=C_H\mu^{2H-2}.
\]
Hence
\[
 \E\left[\mu c(Z_t^\mu)(U_t^{\mu,0})^2\right]
 =C_H\,\E[c(Z_t^\mu)]\,\mu^{2H-1}.
\]
Unless a structural cancellation forces $\E[c(Z_t^\mu)]=0$, this deterministic
bias cannot be improved.  An endpoint proof should replace the
$\gamma$-H\"older bounds in \Cref{lem:variable-convolution} by exact Gaussian
covariance estimates for the frozen fast process, followed by a freezing
error estimate.  This identifies a concrete route to
\Cref{conj:endpoint}.

\begin{remark}[The additive case]
If $\sigma$ is constant, then $c\equiv0$ and the quadratic term vanishes.
The generic obstruction $\mu^{2H-1}$ is absent.  The boundary layer has size
$\mu^H$, reproducing the additive scaling in \cite{Son2020}.  In stable
linear models, the stationary one-time marginal can converge even faster,
which shows that a finite-time upper bound need not be sharp for stationary
laws.
\end{remark}

\begin{remark}[Behavior as $H\downarrow1/2$]
The exponent $2H-1$ tends to zero.  Thus the present multiplicative Young
theory is not uniform across the Brownian threshold.  At $H=1/2$,
$\mu(U^\mu)^2$ has a nonzero limit in Lamperti coordinates; It\^o's formula
then converts that term into the usual Brownian interpretation of the
overdamped equation.  This agrees with the noise-induced-drift mechanisms in
classical small-mass theory \cite{HottovyEtAl2015}.
\end{remark}

\subsection{Further problems}

Several extensions are natural.
\begin{enumerate}[label=\textup{(\arabic*)},leftmargin=2.2em]
\item
\textit{Endpoint rate.}  Prove \Cref{conj:endpoint} by a covariance expansion
of the fast fractional Ornstein--Uhlenbeck field.

\item
\textit{Multidimensional systems.}  Without a global Lamperti transform, one
must estimate the Malliavin matrix of the limiting Young or rough
differential equation.  Uniform ellipticity is the natural first case;
H\"ormander systems require the Gaussian rough-path density theory of
\cite{BaudoinEtAl2016}.

\item
\textit{The regime $H<1/2$.}  The factor $\mu^{2H-1}$ no longer vanishes.
The small-mass approximation must then be formulated with a rough-path lift
and a renormalized correction.  It cannot be obtained by continuity from the
Young regime.

\item
\textit{Velocity-dependent noise.}  If $\sigma=\sigma(Y^\mu)$, even the
additive fast velocity has magnitude $\mu^{H-1}\to\infty$.  The noise
coefficient is therefore evaluated at a diverging argument, and no closed
first-order SK limit exists without an additional scaling such as
$\sigma(\mu^{1-H}Y^\mu)$.  This is a fast-variable homogenization problem,
not a direct extension of the present paper.

\item
\textit{Quantitative equilibration.}  Combine the uniform small-mass bound
with a $\mu$-uniform coupling rate for the kinetic noise-history process to
make the function $r_H$ in \eqref{eq:plateau-estimate} explicit.
\end{enumerate}

\end{document}